\documentclass[11pt]{article}
\usepackage[margin=1in]{geometry}
\usepackage{amsmath,amsthm,amssymb,mathtools}
\makeatletter
\renewenvironment{proof}[1][\proofname]{%
  \par\pushQED{\qed}%
  \normalfont 
  \topsep6\p@\@plus6\p@\relax
  \trivlist
  \item[\hskip\labelsep\bfseries #1\@addpunct{.}]%
}{%
  \popQED\endtrivlist\@endpefalse
}
\makeatother

\usepackage{booktabs,array}
\usepackage{microtype}
\usepackage[hidelinks]{hyperref}

\newtheoremstyle{vance}
  {3pt}   
  {3pt}   
  {\normalfont} 
  {}      
  {\bfseries} 
  {.}     
  {.5em}  
  {}      

\theoremstyle{vance}
\newtheorem{theorem}{Theorem}[section]
\newtheorem{lemma}[theorem]{Lemma}
\newtheorem{corollary}[theorem]{Corollary}
\newtheorem*{theorem*}{Theorem}

\theoremstyle{definition}
\newtheorem{definition}[theorem]{Definition}
\newtheorem{remark}[theorem]{Remark}

\newcommand{\K}{\ensuremath{K}}

\title{All-to-all routing on Kautz digraphs: regular routing beats shortest-paths}
\author{Vance Faber \qquad Noah Streib\\[3pt]
\small Center for Computing Sciences}

\date{}

\begin{document}
\maketitle
\begin{abstract}
We study packet routing in the Kautz digraph $\K(d,D)$, where every ordered pair of distinct vertices is connected by a unique shortest directed path. The regular routing introduced in earlier work schedules all ordered pairs in $\tau(d,D) = (D-1)d^{D-2} + D d^{D-1}$ steps. We show that, for every fixed outdegree $d \ge 2$ and all sufficiently large diameters $D$, no shortest–path routing scheme can match this makespan. More precisely, we prove that $\K(d,D)$ contains an edge whose shortest–path congestion strictly exceeds $\tau(d,D)$ when $D$ is sufficiently large. Our construction uses edge–words drawn from a subset of ternary unbordered square–free words, together with a trimming inequality that propagates large congestion at distance $D$ down to shorter distances.  Computations for $d=2$ and small $D$ show that for all $D \ge 4$ there is an edge in $K(2,D)$ with congestion greater than $\tau(2,D)$.
\end{abstract}

\section{Introduction}

The Kautz digraph $\K(d,D)$ is a diameter–$D$ regular directed graph of outdegree~$d$ in
which every ordered pair $(u,v)$ of distinct vertices is connected by a unique
directed shortest path (\emph{geodesic}).  A \emph{routing scheme} for all-to-all communication assigns to each ordered pair
$(u,v)$ a directed path from $u$ to $v$.  The \emph{makespan} of a routing
scheme is the smallest $T$ for which every path can be scheduled using edge–time
labels in $\{0,1,\dots,T-1\}$ without conflicts.

The \emph{regular} routing scheme introduced in~\cite{FaberStreib} uses only walks of length
$D-1$ and $D$ to route packets.  Counterintuitively, these same \emph{long-path} 
walks were observed to reduce congestion relative to shortest-path routing when 
they were studied experimentally by Li, Lu, and Su~\cite{LiLuSu}.  
In~\cite{FaberStreib} it was shown that the regular routing scheme has makespan 
\[
\tau(d,D) = (D-1)d^{D-2} + D\,d^{D-1}.
\]
A fundamental question is whether shortest–path routing can match this makespan for large $D$.

For an edge $e$ in $\K(d,D)$, let $\mathrm{cong}(e)$ be the number of
geodesics that use~$e$.  If some edge $e$ satisfies 
$\mathrm{cong}(e)>\tau(d,D)$,
then no shortest–path routing scheme can be scheduled within~$\tau(d,D)$
steps; in this sense, regular routing beats all shortest–path schemes.

\medskip
\noindent\textbf{Main result.}
The surprising fact is that such edges always exist for large enough $D$, and moreover they may be
chosen to use only \emph{three} of the $d{+}1$ available symbols:

\begin{theorem*}\label{thm:main}
For every fixed outdegree $d\ge 2$ there exists $D_0(d)$ such that for all
$D\ge D_0(d)$ the Kautz digraph $\K(d,D)$ contains an edge $e$ whose
shortest–path congestion satisfies
\[
\mathrm{cong}(e)\;>\;\tau(d,D).
\]
\end{theorem*}

The proof of the theorem proceeds in two conceptually separate layers.

\medskip
\noindent\textbf{(1) The degree–$2$ case.}
Each edge can be considered as a word of length $D+1$ in an alphabet with $d+1$ symbols.
When $d=2$, the chosen edge–words are ternary words with constraints on certain types of
repetitions.  We identify a family $\mathcal F$ of \emph{$7/4^{+}$-free} words~\cite{CMR_CRT_2019,Lothaire} of edges that are both \emph{unbordered} and have constraints on the size and separation of subword repeats. For these edges we prove
that all members of $\mathcal F$ have
$\mathrm{cong}(e)>\tau(2,D)$ for all sufficiently large~$D$.

\medskip
\noindent\textbf{(2) Unbordered square–free words and the general case.}
Classical results of Currie and Shur~\cite{Currie,Shur} guarantee the existence of
\emph{circular square–free} ternary words (words that are unbordered and square free even after a cyclic rotation) of every length $N\ge 18$. The family of unbordered square free ternary words contains $\mathcal F$ when $d > 2$ but the trimming inequality uses a deficit analysis that is less restrictive. Consequently, in this case the unbordered square-free words have exceptionally large congestion in \emph{every} $\K(d,D)$, proving Theorem~\ref{thm:main}.

\medskip
\noindent\textbf{Structure of the paper.}
Section~\ref{sec:kautz} reviews Kautz digraphs and the unique-geodesic property. Section~\ref{sec:edgewords} examines the repetition properties of edge-words. Section~\ref{sec:trim} shows how the presence of many shortest paths of length $D$ through an edge $e$ force many shortest paths of lessor length through $e$.
Section~\ref{sec:deficits} develops the congestion and deficit framework in detail for the outdegree-$2$ case.
Section~\ref{sec:general-d} combines these ingredients to prove our main
$d=2$ result and also extends the argument to general outdegree $d$.
Section~\ref{sec:computations} presents computational evidence, including exhaustive
data for circular square-free edge-words 
for several diameters tested.
We present several open problems in Section~\ref{sec:open-questions} 
before concluding in Section~\ref{sec:conclusion}.  In Appendix~\ref{sec:smallD} we show the results of computations for
$d=2$ and small $D$, while in Appendix~\ref{sec:construct} we describe an explicit backtracking construction of circular $7/4^{+}$-free words.

\section{Kautz digraphs and unique shortest paths}
\label{sec:kautz}

Let $[n]=\{0,1,2\cdots, n-1\}$.  The Kautz digraph $\K(d,D)$ has vertices
\[
V = \{x_0x_1\cdots x_{D-1} \in [d+1]^D : x_i \ne x_{i+1} \},
\]
and directed edges
\[
x_0x_1\cdots x_{D-1} \to x_1x_2\cdots x_{D-1}x_D, \qquad x_D\in[d+1]\setminus\{x_{D-1}\}.
\]

For vertices $u,v\in V$ we write $\mathrm{dist}(u,v)$ for their directed
distance.

\begin{definition}
\label{def:word-path}
Let $x,y$ be vertices of $\K(d,D)$.
A \emph{word-path}, or \emph{walk-word}, of length $L\ge 0$ from $x$ to $y$ is a word
\[
w_0w_1\cdots w_{D+L-1}
\]
over $[d+1]$ such that:
\begin{enumerate}
\item adjacent letters are distinct (Kautz condition), i.e.\ $w_i\ne w_{i+1}$ for all $i$;
\item the first $D$ letters form $x$:~
  $x = w_0w_1\cdots w_{D-1}$;
\item the last $D$ letters form $y$:~
  $y = w_Lw_{L+1}\cdots w_{L+D-1}$.
\end{enumerate}
Equivalently, the windows of length $D$
\[
w_0w_1\cdots w_{D-1},\;
w_1\cdots w_D,\;
\dots,\;
w_L\cdots w_{L+D-1}
\]
are the successive vertices along a directed walk of length $L$ from $x$ to~$y$.
\end{definition}

\begin{lemma}
\label{lem:overlap-distance}
Let $u,v$ be vertices of $\K(d,D)$, and let
\[
\operatorname{ov}(u,v)
:= \max\Bigl\{\, j : 0 \le j \le D,\;
\text{the last $j$ letters of $u$ equal the first $j$ letters of $v$} \Bigr\}.
\]
Then the directed distance between $u$ and $v$ is
\[
\mathrm{dist}(u,v) \;=\; D - \operatorname{ov}(u,v).
\]
Moreover, there is a unique directed path from $u$ to $v$ of length $\mathrm{dist}(u,v)$.
\end{lemma}
\begin{proof}
This is a standard exercise and follows easily by using the word-path determined by overlap of length $j$.  
In this case, not repeating the overlap creates a (unique) word-path of length $D-j$, which corresponds to 
the unique path from $u$ to $v$ in $\K(d,D)$ of length $\mathrm{dist}(u,v)$.
\end{proof}

\section{Edge-words and subword repetition properties}
\label{sec:edgewords}

In this section, we focus on \textit{ternary} words, vertices in $\K(2,D)$. By Definition~\ref{def:word-path}, every directed edge $e : u \to v$ determines a word-path of length $D+1$,
\[
w(e)=a_0a_1\cdots a_D,
\qquad
u=a_0\cdots a_{D-1}, \;
v=a_1\cdots a_D,
\]
with the Kautz constraint $a_i\ne a_{i+1}$.

\begin{definition}[Border and unbordered word]
A word $a_0\cdots a_D$ has a \emph{border} of length $\ell$ if
\[
a_0a_1\cdots a_{\ell-1}
=
a_{D-\ell+1}\cdots a_D.
\]
It is \emph{unbordered} if it has no border of length $1\le \ell\le D$.
\end{definition}

Intuitively, unborderedness 
prevents the short directed cycles that would otherwise restrict the
growth of geodesics through $e$.

\begin{definition}[$\alpha$-power, $\alpha$-free, and $\alpha^{+}$-free]
Let $W$ be a word and let $\alpha>1$.  We say that $W$ contains an
\emph{$\alpha$-power} if $W$ has a subword of the form $x^k y$,
where $x$ is non-empty, $k\ge 1$ is an integer, and $y$ is a (possibly empty)
prefix of $x$, such that
\[
\frac{|x^k y|}{|x|} \;\ge\; \alpha.
\]
Equivalently, writing $|x^k y|=k|x|+|y|$, this condition is
$|y|/|x|\ge\alpha-k$.
We say that $W$ is \emph{$\alpha$-free} if it contains no $\alpha$-power.
We say that $W$ is \emph{$\alpha^{+}$-free} if it contains no subword $x^k y$
with $\frac{|x^k y|}{|x|}>\alpha$.
\end{definition}

\begin{definition}[Circular $\alpha$-power-free]
A word $W$ is \emph{circularly $\alpha$-free} (respectively \emph{circularly $\alpha^+$-free})
if every cyclic rotation of $W$ is $\alpha$-free (respectively $\alpha^+$-free).
\end{definition}

\begin{remark}
A \emph{square} is exactly a $2$-power of the form $XX$, 
so \emph{square-free} is the same as \emph{$2$-free}.
More generally, if $\alpha\le \beta$ then $\alpha$-free implies $\beta$-free (and likewise
$\alpha^+$-free implies $\beta^+$-free), since forbidding larger exponents is weaker; 
e.g., square-free is strictly weaker than $7/4^{+}$-free. A circularly $\beta$-free word 
with $\beta\le 2$ is automatically unbordered: a nontrivial
border forces a cyclic rotation beginning with a square $PP$ (exponent $2$).  
In the ternary circular setting,
circular square-free words exist for $n\notin\{5,7,9,10,14,17\}$ \cite{CurrieJohnson}, whereas circular
$7/4^{+}$-free words fail additionally at $n\in\{16,22\}$ \cite{CMR_CRT_2019}.
Thus, the smallest length where circular square-free words exist but circular 
$7/4^{+}$-free words do not is $n=16$ (i.e.\ $D=15$).
For example,
\[
w=0121020102120102
\]
is a circular square-free ternary word of length $16$ but it is not $7/4^{+}$-free. 
Indeed, the subword $10201210201$ has the form $x^1y$ with $x=102012$ and $|y|=5$, 
so its repetition exponent is $|x^1y|/|x|=1+5/6>7/4$.
Despite this obstruction, computational checks indicate that even at $n=16$ the
circular square-free edge-words already exhibit congestion well above the makespan,
so our sufficient conditions are not expected to be sharp at small $D$.
\end{remark}

In this paper, we utilize the two regimes $\alpha=2$ (square-free) and $\alpha=7/4$
($7/4^{+}$-free). We will show that for a given $d$, every circular $7/4^{+}$-free edge 
with large enough $D$ has congestion greater than $\tau(d,D)$.

\section{A trimming inequality for geodesic layers}
\label{sec:trim}

Fix an oriented edge $e$ in $\K(d,D)$ and integers $1\le k\le D$, $1\le t\le k$.

\begin{definition}
Let $N(e;k,t)$ denote the number of ordered vertex pairs $(x,y)$ such that:
\begin{enumerate}
\item $\mathrm{dist}(x,y)=k$, and
\item the unique geodesic from $x$ to $y$ traverses $e$ in position $t$.
\end{enumerate}
Define
\[
U_k(e):=\sum_{t=1}^k N(e;k,t),
\qquad
\mathrm{cong}(e)=\sum_{k=1}^D U_k(e).
\]
\end{definition}

For fixed $k$ and $t$, the maximum possible value of $N(e;k,t)$ is $d^{k-1}$:
once the interior $(D+1)$-letter window corresponding to the edge-word is fixed,
the remaining letters in $x$ and $y$ each contribute $d$ choices. We need to investigate how far $N(e;D,t)$ falls short of this
maximum. But first, we quantify the observation that if an edge is on many geodesics of length $D$ then it must be on many geodesics of shorter lengths also. 

\begin{lemma}[Trimming inequality]\label{lem:UDm1-from-UD}
For any directed edge $e$ in $\K(d,D)$,
\[
U_{D-1}(e)\ \ge\ \frac{D-1}{dD}\,U_D(e).
\]
\end{lemma}

\begin{proof}
Write $a_t := N(e;D,t)$ 
and $b_t := dN(e;D-1,t)$, for $t=1,\dots,D-1$.

\smallskip
\noindent\emph{Forward trimming.}
Deleting the last edge of a length-$D$ geodesic using $e$ in position $t$
produces a geodesic of length $(D-1)$ using $e$ in position $t$, and each
geodesic of length $(D-1)$ extends forward in at most $d$ ways.  
Thus $a_t \le b_t$.

\smallskip
\noindent\emph{Backward trimming.}
Deleting the first edge of a length-$D$ geodesic using $e$ in position $t+1$
produces a geodesic of length $(D-1)$ using $e$ in position $t$, and each such
geodesic extends backward in at most $d$ ways.  
Thus $a_{t+1} \le b_t$.

\smallskip
Therefore
\[
b_t \;\ge\; \max(a_t, a_{t+1})\qquad (t=1,\dots,D-1).
\]
Summing gives
\[
dU_{D-1}(e)
=\sum_{t=1}^{D-1} b_t
\;\ge\; \sum_{t=1}^{D-1} \max(a_t,a_{t+1}).
\]

Among nonnegative $a_1,\dots,a_D$ with fixed sum 
$S=\sum_{t=1}^D a_t = U_D(e)$, the quantity
$\sum_{t=1}^{D-1} \max(a_t,a_{t+1})$ is minimized when all $a_t$ are equal,
in which case it equals $(D-1)(S/D)$.

Thus
\[
dU_{D-1}(e)
\;\ge\; \frac{D-1}{D}\,U_D(e),
\]
which is the desired inequality.
\end{proof}

The same argument works at all lengths and yields a telescoping bound.

\begin{lemma}\label{lem:Uk-chain}
Let $e$ be a directed edge in a digraph of indegree and outdegree $2$.
Then for every integer $k\ge2$,
\[
U_{k-1}(e)\ \ge\ \frac{k-1}{dk}\, U_k(e).
\]
\end{lemma}

\begin{proof}
Identical to Lemma~\ref{lem:UDm1-from-UD}, with $D$ replaced by $k$.
\end{proof}

Iterating Lemma~\ref{lem:Uk-chain} down from length $D$ yields a lower
bound on the \emph{total} congestion in terms of the number of geodesics through $e$ of length $D$ alone.

\begin{corollary}\label{cor:sumUk-from-UD}
For any edge $e$ in $\K(d,D)$,
\[
U_k(e)
\ \ge\
\frac{k}{D}\,d^{-(D-k)}\,U_D(e)
\qquad\text{for every }1\le k\le D,
\]
and hence

\[
\sum_{k=1}^D U_k(e)\ \ge\ \frac{d}{d-1}\left(1-\frac{1}{D(d-1)}\right)\,U_D(e).
\]

\end{corollary}

\begin{proof}

Iterating Lemma~\ref{lem:Uk-chain} gives
\[
U_k(e)
\ \ge\
\Biggl(\prod_{m=k+1}^D \frac{m-1}{dm}\Biggr) U_D(e)
\qquad (1\le k\le D).
\]
The product telescopes:
\[
\prod_{m=k+1}^D \frac{m-1}{m}
=\frac{k}{D},
\qquad
\prod_{m=k+1}^D \frac1{d}
=d^{-(D-k)},
\]
so $U_k(e)\ge \frac{k}{D}d^{-(D-k)}U_D(e)$. Hence

\begin{equation*}\label{eq:sumUk-lower}
\sum_{k=1}^D U_k(e)
\ \ge\
\frac{U_D(e)}{D}\sum_{k=1}^D k\,d^{-(D-k)}
\;=\;
\frac{U_D(e)\Bigl(d^2 - d - \frac{d}{D} + \frac{d^{1-D}}{D}\Bigr)}{(d-1)^2}
\end{equation*}
and dropping the last term in the numerator proves the second inequality.
\end{proof}

\section{Deficits at distance D and witness templates}
\label{sec:deficits}

In this section, we assume $d=2$ unless otherwise stated.

\begin{definition}[Deficits at length $D$]\label{def:deficit}
For a fixed edge $e$ and a fixed position $t\in\{1,\dots,D\}$, define the
\emph{deficit}
\[
\Delta_t(e)
\;:=\;
2^{D-1} - N(e;D,t),
\]
and write
\[
\Delta(e)=\sum_{t=1}^D \Delta_t(e)
\]
for the total deficit at length $D$.  Then
\[
U_D(e)
=
\sum_{t=1}^D N(e;D,t)
=
D\,2^{D-1}-\Delta(e).
\]
\end{definition}

Fix an edge $e$ with edge-word $a_0a_1\cdots a_D$.
If a geodesic of length $D$ uses $e$ in position $t\in\{1,\dots,D\}$, then it is
encoded by a word-path
\[
W = w_0w_1\cdots w_{2D-1}
\]
of length $D$ from some source vertex $x=w_0\cdots w_{D-1}$ to some destination vertex $y=w_D\cdots w_{2D-1}$, in the
sense of Definition~\ref{def:word-path}, in which the block $a_0\cdots a_D$
appears as a contiguous subword.

More concretely, for each position $t$ there is a unique placement of the
edge-window $a_0\cdots a_D$ inside $W$ such that $a_0$ is the $t$th letter of
the walk.  Let $\ell(t)=t-1$ be the number of letters of $W$
that precede $a_0$ and $r(t)=D-t$ the number of letters that follow $a_D$.
We call these $\ell(t)$ and $r(t)$ letters the \emph{left} and \emph{right flanks} of $e$ at position $t$, respectively.

\subsection{General position $t$ and overlap $r$}

Fix a diameter $D\ge2$ and work in $\K(2,D)$.
Let $e$ be a fixed edge whose edge-word is $w(e)=a_0a_1\cdots a_D$,
and fix a position $t\in\{1,\dots,D\}$.
Consider a walk-word
$W=w_0 w_1 \cdots w_{2D-1}$
of length $2D$ corresponding to a (not necessarily geodesic) walk of length $D$
from $x$ to $y$, with
$x = w_0 w_1 \cdots w_{D-1}$, and
$y = w_D w_{D+1} \cdots w_{2D-1}$.
We suppose that this walk uses $e$ in position $t$, i.e., the edge-word window
$a_0 a_1 \cdots a_D$
appears in $W$ starting at index $t-1$, so the edge-window is
\[
I_E \;=\; [\,t-1,\;D+t-1\,],
\]
an interval of length $D+1$.

Now suppose we are looking at a non-geodesic pair $(x,y)$, so that
the overlap length $r:=\operatorname{ov}(x,y)\ge2$.
Thus, $w_{D-r}\cdots w_{D-1}=w_D\cdots w_{D+r-1}$.
We denote the two intervals that will overlap by
\[
I_L = [\,D-r,\;D-1\,],
\qquad
I_R = [\,D,\;D+r-1\,],
\]
and write
\[
J = I_L \cup I_R = [\,D-r,\;D+r-1\,]
\]
for the union of the two copies of the overlapping block.

\medskip
\noindent
\textbf{Excluding the square-inside-edge case.}
Assume throughout that $w(e)$ is unbordered and square-free. If both copies of the overlap 
$I_L$ and $I_R$ lay entirely inside the edge-window $I_E$, then
$w(e)$ would contain a nontrivial square $\beta\beta$ of length $2r$,
contradicting square-freeness.
Hence $J \not \subseteq I_E$:
at least one endpoint of $J$ must lie outside the edge-window.

In terms of inequalities, this gives two possibilities:
\begin{align*}
\text{(L)}\quad&D-r < t-1
\qquad\Longleftrightarrow\qquad t+r > D+1,\\[3pt]
\text{(R)}\quad&D+r-1 > D+t-1
\qquad\Longleftrightarrow\qquad r > t.
\end{align*}
Thus, for each non-geodesic pair $(x,y)$ contributing a deficit at position $t$
and having overlap $r\ge2$, we must be in at least one of the two cases:
$t+r > D+1$ or 
$r>t$.

\medskip
\noindent
\textbf{Balancing position $t$ around the midpoint.}
The position $t$ can be ``left-heavy'' or ``right-heavy'' depending on how the
edge-window $I_E=[t-1,D+t-1]$ sits around the midpoint $D$:
\[
2t \;\le\; D+1
\quad\Longleftrightarrow\quad
t-1 \;\le\; D-t,
\]
in which case the edge extends more to the left than to the right of the midpoint,
or
\[
2t \;>\; D+1,
\]
in which case it extends further to the right.
These two possibilities are symmetric under reversing the walk
\(W \mapsto w_{2D-1}\cdots w_0\), so it suffices to analyze one of them in
detail and obtain the other by left--right symmetry.  In the remainder of this
subsection we treat the case
\[
2t \;\le\; D+1,
\]
i.e., the edge-window is more to the left than the right of the midpoint.

\subsubsection*{Case A: $2t\le D+1$ and $t+r>D+1$}

Assume $2t\le D+1$ and $t+r>D+1$ (equivalently, $r>D-t+1$).  Since $2t\le D+1$
implies $D-t+1\ge t$, we in fact have $r>t$ as well.  Thus the overlap interval
$J=[D-r,\,D+r-1]$ extends past \emph{both} ends of the edge-window
$I_E=[t-1,\,D+t-1]$.

In this regime the overlap condition forces $w(e)$ to be a subword crossing the
midpoint of a square, hence $w(e)$ has a nontrivial border.  Therefore Case~A is
impossible whenever $w(e)$ is unbordered (see Lemma~\ref{lem:case-A-forbidden}).

\subsubsection*{Case B: $2t\le D+1$ and $r>t$}

Now assume
\[
2t \;\le\; D+1,
\qquad
r>t,
\qquad
t+r \;\le\; D+1,
\qquad
r\ge2.
\]
From $t+r\le D+1$ we obtain $D-r \ge t-1$, so
\[
t-1 \;\le\; D-r \;\le\; D-1 < D < D+t-1 < D+r-1 \le 2D-1.
\]
Thus, the left copy of the overlap lies entirely inside the left half of the edge:
  $I_L \subseteq [\,t-1,\;D-1\,] \subseteq I_E$, and
the right copy of the overlap starts inside the edge but protrudes to the
right:
  $I_R \cap I_E = [\,D,\;D+t-1\,]$ and
  $I_R\setminus I_E = [\,D+t,\;D+r-1\,]$.

The relevant cut points on $[0,2D-1]$ are
\[
0,\ t-1,\ D-r,\ D-1,\ D,\ D+t-1,\ D+r-1,\ 2D-1,
\]
which give rise to the corresponding subwords of $W$:
\begin{align*}
\mu      &:= w_0 w_1 \cdots w_{t-2}        &&\text{(far left prefix)},\\
A        &:= w_{t-1} \cdots w_{D-r-1}      &&\text{(part of $x$ before the left overlap)},\\
U        &:= w_{D-r} \cdots w_{D-1}        &&\text{(left copy of the overlap block)},\\
B        &:= w_{D} \cdots w_{D+t-1}        &&\text{(portion of the right copy inside the edge)},\\
V        &:= w_{D+t} \cdots w_{D+r-1}      &&\text{(portion of the right copy outside the edge)},\\
\chi     &:= w_{D+r} \cdots w_{2D-1}       &&\text{(far right suffix)}.
\end{align*}

By construction, $U$ and $BV$ are the two equal overlap blocks of length $r$,
so
\[
U = BV.
\]
The walk-word therefore factors as
\[
W = \mu\; A\; U\; B\; V\; \chi
     \;=\;
\mu\; A\; (BV)\; B\; V\; \chi.
\]
The edge-window $I_E = [t-1,D+t-1]$ covers exactly the concatenation of $A$,
$U$, and $B$:
\[
w = a_0\cdots a_D
  = w_{t-1}\cdots w_{D+t-1}
  = A\,U\,B
  = A\,(BV)\,B.
\]
Here $U$ has length $r\ge2$ and $B$ has length $t$, while
\[
|A| = (D-r)-(t-1) = D+1-t-r \;\ge\;0,
\]
with $A$ possibly empty only at the boundary $t+r=D+1$.

\begin{lemma}\label{lem:case-A-forbidden}
In Case~A ($2t\le D+1$ and $t+r>D+1$), the edge-word $w(e)$ cannot be unbordered.
\end{lemma}
\begin{proof}
Let
$U:=w_{D-r}\cdots w_{D-1}$.
By the definition of an overlap of length $r$, we have
\[
w_{D-r+i}=w_{D+i}\qquad(0\le i\le r-1),
\]
so the subword $w_{D-r}\cdots w_{D+r-1}$ equals the square $UU$.
In Case~A we have $D-r<t-1$ (equivalently $t+r>D+1$) and also $r>t$ (since
$2t\le D+1$ implies $D-t+1\ge t$ and $r>D-t+1$), hence
\[
t-1\ \ge\ D-r+1
\qquad\text{and}\qquad
D+t-1\ \le\ D+r-2.
\]
Therefore the entire edge-window $[t-1,\,D+t-1]$ lies inside $[D-r,\,D+r-1]$, so
$w(e)$ is a subword of $UU$.  In particular $w(e)$ has period $r$.

Since $r\le D$, we have $|w(e)|=D+1>r$, and any word of length $L$ with period $r<L$
has a nontrivial border of length $L-r$.  Thus $w(e)$ is bordered, contradicting
the assumption that $w(e)$ is unbordered.
\end{proof}

Consequently, for an unbordered square-free edge-word and $2t\le D+1$, every
non-geodesic pair contributing a deficit at position~$t$ arises from Case~B.
By left--right symmetry of the walk-word $W\mapsto w_{2D-1}\cdots w_0$, the
same structural description applies for $2t>D+1$ with the roles of the left and
right flanks exchanged.

\subsection{Witness templates and cost}

We now extract from Case~B a ``template'' and a concrete cost
function.

Fix $t$ with $2t\le D+1$ and an overlap length $r\ge2$ that can occur at
position~$t$ for a non-geodesic pair $(x,y)$. The analysis in 
Case~B gives a factorization
\[
W = \mu\,A\,U\,B\,V\,\chi,
\qquad
U = B V,
\]
with
\[
w(e) = a_0\cdots a_D = A\,U\,B = A\,(BV)\,B,
\]
where $|B|=t$ and $|U|=r$. We also have
\[
2t \;\le\; D+1,
\qquad
r>t,
\qquad
t+r \;\le\; D+1,
\qquad
r\ge2
\]
and the overlap interval
$J = [\,D-r,\;D+r-1\,]$
meets the flanks only in the segment $[D+t,D+r-1]$, corresponding to the
subword $V$ in the right flank.  Thus the letters of $V$ are uniquely
determined by $w(e)$ and the requirement that the two copies of the overlap
block coincide, while all other flank letters remain free apart from the usual
Kautz adjacency constraint.

In particular,
\[
|V| = r-t
\]
is exactly the number of flank letters that are \emph{forced} by the overlap
pattern at $(t,r)$.

\begin{definition}[Witness template and cost]
\label{def:template-cost}
Fix an edge $e$ with unbordered square-free edge-word $w(e)$ and a position
$t$ such that $2t\le D+1$.

\begin{itemize}
\item A \emph{witness at $(t,r)$} is a non-geodesic pair $(x,y)$ whose
      overlap length is $\operatorname{ov}(x,y)=r\ge2$ and whose walk-word $W$
      factors as
      \[
      W = \mu\,A\,U\,B\,V\,\chi,
      \qquad
      U = B V,
      \qquad
      w(e) = A U B = A(BV)B,
      \]
      with $|B|=t$ and $|U|=r$, as in Case~B.

\item The \emph{witness template} $\mathcal T(t,r)$ is         the set obtained by recording
      which flank positions lie in the overlap
      interval $J$ and hence are forced, and which flank positions remain free.

\item The \emph{cost} of the template is
      \[
      m(\mathcal T(t,r)) := m(t,r) = r-t,
      \]
      i.e., the number of flank letters whose values are uniquely determined by
      $w(e)$ and the equality of the two overlap copies.
\end{itemize}
\end{definition}

Because the edge-word $w(e)$ is fixed, once $(t,r)$ is fixed the factorization
$w(e) = A(BV)B$ (if it exists at all) is unique: $B$ is the suffix of $w(e)$ of
length $t$, and $BV$ is the suffix of length $r$.  Thus for each $(t,r)$ there
is \emph{at most one} template $\mathcal T(t,r)$.


\begin{definition}[Admissible overlap lengths at position $t$]\label{def:Rt-word}
Let $w$ be a word of length $n$ and fix $t\ge 1$.
Write $B$ for the suffix of $w$ of length $t$. We say that an integer $r$ with $t+1\le r\le n-t-1$ is \emph{admissible at position $t$}
if there exist words $A,V$ such that $V\neq\epsilon$, where $\epsilon$ denotes the empty word, and
\begin{equation}\label{eq:Rt-overlap}
w \;=\; A\,B\,V\,B,
\qquad |B|=t,
\qquad |BV|=r.
\end{equation}
Define
\[
R_t(w)\ :=\ \bigl\{\,r\in\{t+1,\dots,n-t-1\}:\ r \text{ is admissible at position }t\,\bigr\}.
\]

\smallskip
\noindent
(In the applications below, $w$ will be an edge-word $w(e)$ of length $n=D+1$.
The factorization \eqref{eq:Rt-overlap} is exactly the ``$ABVB$'' overlap pattern
arising from a witness template $\mathcal T(t,r)$, with \emph{cost} $m=r-t=|V|$.)
\end{definition}

\subsection{Example}

Before continuing with the proof, we pause to provide an illustrative example
of the concepts developed in this section.

Let $u$ and $v$ be vertices in $K(2,7)$ with $u=0120210$ and $v=1202102$ such that
$e=uv$ is an edge in $E(K(2,7))$  with edge-word
\[
w(e) \; = \; 01202102. 
\]
Note that $w(e)$ is circular square-free.  

Let $t=2$.  Consider the walk-word $W$ of length $D$ with these parameters from some unspecified 
vertex $x$ to another unspecified vertex $y$:
\[
W \; = \; \cdot \; 0 \; 1 \; 2 \; 0 \; 2 \; 1 \; | \; 0 \; 2 \; \cdot \; \cdot \; \cdot \; \cdot \; \cdot 
\]
where ``$\cdot$'' denotes an unspecified character and ``$|$'' denotes the midway point, i.e., it separates 
the string for $x$ from the string for $y$.

We can satisfy the conditions of Case B by selecting $r \in \{3,4,5,6\}$.
However, in this example, the only admissible choice from this set is $r=3$ (and hence $J = [4,9]$), which we get 
by setting $B=02$ and $V=1$ (and $A=012$).  That is, there are $|V|=1$ flank positions that 
are forced in order to produce a witness at $(2,3)$ (the set of characters $\{w_9\}$)---which must be the string $1$
to force the repetition of the string $021$ in $W$---while the remaining
unspecified characters (the set $\{w_0,w_{10},w_{11},w_{12},w_{13}\}$) remain free.  The witness
template $\mathcal T(2,3)$ has cost $m(t,r)=1$, and the number of possible witnesses $(x,y)$ for this
template is at most $2^5$.

\subsection{Bounding Lemmas}

We now prove a sequence of lemmas that will be useful 
in Sections~\ref{sec:general-d} and~\ref{sec:computations}.

\begin{lemma}[Strong separation of admissible overlap lengths for fixed $t$]
\label{lem:rt-separation}
Let $w$ be a square-free word of length $n$, and fix $t\ge 1$.
If $r,r'\in R_t(w)$ are distinct, then
\[
|r-r'|\ \ge\ t+1.
\]
In particular, $R_t(w)$ contains no two consecutive integers.
\end{lemma}

\begin{proof}
Fix $t$ and let $B$ be the suffix of $w$ of length $t$.
Choose $r=\min R_t(w)$, so by Definition~\ref{def:Rt-word} we can write
\[
w = A\,B\,V\,B
\]
with $V\neq\epsilon$, $|B|=t$, and $|BV|=r$.

Now let $r'\in R_t(w)$ with $r'>r$.  The corresponding occurrence of the same suffix
$B$ starts $r'-r$ positions earlier than the occurrence used above, so we may write
\[
w = A'\,B\,X\,V\,B
\]
where $|X|=r'-r$ (i.e.\ the gap between these two starts is $|X|$).

Assume for contradiction that $|X|\le t$.  Since $w=A\,B\,V\,B$, the word $B\,V\,B$
is a suffix of $w$.  Because $w=A'BXV B$, the word $XVB$ is also a suffix of $w$,
hence $XVB$ is a suffix of $BVB$.  If $|X|\le t$, this suffix begins inside the
leading $B$ in $BVB$, so $X$ is a suffix of $B$.  Write $B=Y X$.
Then $B X = Y X X$ contains the square $XX$ as a factor of $w$, contradicting that
$w$ is square-free.

Therefore $|X|\ge t+1$, i.e.\ $r'-r\ge t+1$.  Since $r<r'$ were arbitrary distinct
elements of $R_t(w)$, the claim follows.
\end{proof}

\begin{lemma}
\label{lem:template-contribution}
Fix an edge $e$ and a position $t$ with $2t\le D+1$ and
\[
R_t(e)=\Bigl\{\,r\in\{t+1,\dots,D-t\}:\text{a witness template }\mathcal T(t,r)\text{ exists}\Bigr\}.
\]
For each $r\in R_t(e)$, the corresponding template has cost $m=r-t$, and at most
$2^{D-1-r+t}=2^{D-1-m}$ ordered pairs $(x,y)$ realize this template. Hence
\[
0\le \Delta_t(e;\mathcal T(t,r)) \le 2^{D-1-r+t},
\qquad\text{and consequently}\qquad
\Delta_t(e)\le \sum_{r\in R_t(e)} 2^{D-1-r+t}.
\]
\end{lemma}

\begin{proof}
Fix $r\in R_t(e)$ and let $\mathcal T=\mathcal T(t,r)$ be a witness template.
By definition of the template parameters, the overlap interval (or forced window)
has length $m=r-t$, so the \emph{cost} of $\mathcal T$ is $m=r-t$.

Now count the number of ordered pairs $(x,y)$ that realize $\mathcal T$.
Realizing $\mathcal T$ means that the wide-word (equivalently, the pair $(x,y)$)
agrees with the template on every position that is forced by $\mathcal T$.
The template $\mathcal T(t,r)$ fixes all coordinates except for the
$D-1-r+t$ ``free'' coordinates outside the forced window, and each free coordinate
has two choices (binary choice of the next letter in the two-symbol subalphabet,
or equivalently the two admissible continuations in the Kautz constraint).
Therefore the number of realizations is at most
$2^{D-1-r+t}$.
Each realization contributes at most $1$ to the template-specific deficit
$\Delta_t(e;\mathcal T)$ and contributes nonnegatively, hence
\[
0\le \Delta_t(e;\mathcal T)\le 2^{D-1-r+t}.
\]
Summing over all $r\in R_t(e)$ gives the stated bound on $\Delta_t(e)$.
\end{proof}

\begin{lemma}
\label{lem:sparse-geometric}
Fix an edge $e$ and an integer $t$ with $2t\le D+1$.
Assume $w(e)$ is square-free. If $R_t(w(e))\neq\emptyset$, set
$r_0=\min R_t(w(e))$.
Then
\[
\Delta_t(e)\ \le\ \frac{1}{1-2^{-(t+1)}}\,2^{D-1-r_0+t}.
\]
\end{lemma}

\begin{proof}
List the elements of $R_t(w(e))$ in increasing order as
$r_0<r_1<r_2<\cdots$.  By Lemma~\ref{lem:rt-separation} we have
$r_{k+1}-r_k\ge t+1$ for all $k$, and hence
\[
r_k \ge r_0 + k(t+1)\qquad (k\ge 0).
\]
Therefore
\[
\sum_{r\in R_t(w(e))}2^{D-1-r+t}
=\sum_{k\ge 0}2^{D-1-r_k+t}
\le \sum_{k\ge 0}2^{D-1-(r_0+k(t+1))+t},
\]
which is a geometric series with ratio $2^{-(t+1)}$
giving the stated bound.
\end{proof}

\begin{lemma}
\label{lem:weighted-sparsity}
For an edge $e$ of $\K(2,D)$ and $t$ with $2t\le D+1$ and
\[
R_t(e)=\{\,r\in\{t+1,\dots,D-t\}:\ \text{a witness template }\mathcal T(t,r)\text{ exists}\,\},
\]
define the weighted sparsity
\[
\Omega(e):=\sum_{t=1}^{\lceil (D+1)/2\rceil}\ \sum_{r\in R_t(e)} 2^{-(r-t)}.
\]
Then
\[
\Delta(e)\le 2^{D}\Omega(e),
\qquad\text{and hence}\qquad
U_D(e)\ge (D-2\Omega(e))\,2^{D-1}.
\]
\end{lemma}

\begin{proof}
Fix $t\le \lceil (D+1)/2\rceil$.  A witness of type $(t,r)$ has cost $s=r-t\ge 1$,
meaning it fixes $s$ flank symbols and leaves $D-1-s$ flank symbols free.
Thus the number of completions contributing to $\Delta_t(e)$ from a fixed
$(t,r)$ is at most $2^{D-1-(r-t)}$. Summing over admissible $r$ gives
\[
\Delta_t(e)\le \sum_{r\in R_t(e)} 2^{D-1-(r-t)}
=2^{D-1}\sum_{r\in R_t(e)}2^{-(r-t)}.
\]
Summing over $t\le \lceil (D+1)/2\rceil$ and using the symmetry
$\Delta_t(e)=\Delta_{D+1-t}(e)$ yields $\Delta(e)\le 2^D \Omega(e)$.
Since $U_D(e)=D2^{D-1}-\Delta(e)$, the stated lower bound on $U_D(e)$ follows.

\end{proof}

\begin{lemma}
\label{lem:sufficient-W}
If $D>3$ and
\[
\Omega(e)\ <\ \frac{D-3}{8},
\]
then $\mathrm{cong}(e)>\tau(2,D)$.
\end{lemma}

\begin{proof}
By Lemma~\ref{lem:weighted-sparsity},
\[
U_D(e)\ \ge\ (D-2\Omega(e))\,2^{D-1}.
\]
By Corollary~\ref{cor:sumUk-from-UD} specialized to $d=2$,
\[
\mathrm{cong}(e)=\sum_{k=1}^D U_k(e)\ \ge\ \Bigl(2-\frac{2}{D}\Bigr)U_D(e)
\ge\ \frac{D-1}{D}\,(D-2\Omega(e))\,2^{D}.
\]
Since $\tau(2,D)=(D-1)2^{D-2}+D2^{D-1}=(3D-1)2^{D-2}$, the inequality
$\mathrm{cong}(e)>\tau(2,D)$ follows if
\[
4 \frac{D-1}{D}\,(D-2\Omega(e))\ >\ 3D-1.
\]
This is satisfied if $\Omega(e)<\dfrac{D(D-3)}{8(D-1)}$, meaning that it is also
satisfied when $\Omega(e)<\dfrac{D-3}{8}$.
\end{proof}

\section{High congestion for general outdegree \texorpdfstring{$d$}{d}}
\label{sec:general-d}

Before proving the main theorem, we extend the deficit--plus--trimming strategy from $\K(2,D)$
to the Kautz digraph $\K(d,D)$ of outdegree $d\ge 2$.
Our main point is that sufficiently repetition-free \emph{ternary} edge-words
already force high shortest-path congestion in every $\K(d,D)$:
we may view the ternary alphabet $\{0,1,2\}$ as a subset of $[d{+}1]$ and use the
same edge-word for all $d$.

\subsection{Edge-words and overlaps in $\K(d,D)$}

Vertices of $\K(d,D)$ are words of length $D$ over the alphabet $[d{+}1]$
with no repeated adjacent letters.  Each vertex has outdegree $d$:
from $u=u_0\cdots u_{D-1}$ one may append any symbol in $[d{+}1]\setminus\{u_{D-1}\}$.
Every edge $e:u\to v$ has an edge-word
\[
w(e)=a_0a_1\cdots a_D,
\qquad
u=a_0\cdots a_{D-1},\;
v=a_1\cdots a_D.
\]

As in the degree--$2$ case, for a fixed edge $e$ and position $j\in\{1,\dots,D\}$,
choosing the $(D-1)$ flank symbols determines a length-$D$ walk that uses $e$
in position $j$.  There are at most $d^{D-1}$ such flank choices and we write
\[
\Delta_j(e) := d^{D-1}-N(e;D,j),
\qquad
\Delta(e) := \sum_{j=1}^D \Delta_j(e),
\]
where $N(e;D,j)$ denotes the number of ordered pairs $(x,y)$ at distance $D$
whose chosen length-$D$ geodesic uses $e$ in position $j$.  The deficit $\Delta_j(e)$ is 
contributed to by non-geodesic flank
choices, i.e., those for which the resulting pair has overlap length $r\ge 2$.

For $j$ with $2j\le D+1$, let
\[
R_j(e):=\{\,r\in\{j+1,\dots,D-j\} : \text{a witness template }\mathcal T(j,r)\text{ exists}\,\}.
\]
Exactly as in Lemma~\ref{lem:template-contribution}, the contribution of an overlap of length $r$ is controlled
by its cost $m=r-j\ge 1$: a cost-$m$ witness forces $m$ flank symbols and
leaves $D-1-m$ flank symbols free.  Hence, for each fixed $j$,
\[
\Delta_j(e)\ \le\ \sum_{r\in R_j(e)} d^{D-1-(r-j)}.
\]

\begin{definition}[Weighted overlap count in $\K(d,D)$]
\label{def:Wd}
Define
\[
\Omega_d(e)
\;:=\;
\sum_{j=1}^{\lceil (D+1)/2\rceil}\ \sum_{r\in R_j(e)} d^{-(r-j)}.
\]
\end{definition}

\begin{lemma}
\label{lem:deficit-from-Wd}
Let $e$ be an edge in $\K(d,D)$ whose edge-word is square-free.
Then
\[
\Delta(e)\ \le\ 2\,d^{D-1}\,\Omega_d(e),
\qquad\text{and hence}\qquad
U_D(e)\ \ge\ \bigl(D-2\Omega_d(e)\bigr)\,d^{D-1}.
\]
\end{lemma}

\begin{proof}
Fix $j$ with $2j\le D+1$.  From the cost bound above,
\[
\Delta_j(e)\ \le\ \sum_{r\in R_j(e)} d^{D-1-(r-j)}
= d^{D-1}\sum_{r\in R_j(e)} d^{-(r-j)}.
\]
Summing over $j\le \lceil(D+1)/2\rceil$ and doubling by the symmetry $j\mapsto D+1-j$
gives
\[
\Delta(e)\le
2\,d^{D-1}\sum_{j=1}^{\lceil (D+1)/2\rceil}\ \sum_{r\in R_j(e)} d^{-(r-j)}
=
2\,d^{D-1}\,\Omega_d(e).
\]
Substituting into $U_D(e)=Dd^{D-1}-\Delta(e)$ yields the stated bound.
\end{proof}

\subsection{A universal cost bound from \texorpdfstring{$7/4^{+}$}{7/4+}-freeness}

The key combinatorial input for general $d$ is a \emph{lower bound on the cost}
$r-j$ of any admissible overlap at position $j$.
This is where we use $7/4^{+}$-power-freeness.

\begin{lemma}
\label{lem:74-cost}
Let $w$ be a $7/4^{+}$-power-free word.  Suppose $w$ has a subword of the form
$BVB$ with $|B|=t\ge 1$ and $|V|=m\ge 1$.  Then $m\ge \lceil t/3\rceil$.
\end{lemma}

\begin{proof}
Set $x:=BV$, so $|x|=t+m$, and observe that $BVB = x\cdot B$ where $B$ is a prefix of $x$.
Thus $BVB$ is an $\alpha$-power of period $|x|$ with
\[
\alpha = \frac{|BVB|}{|x|}
      = \frac{2t+m}{t+m}
      = 1 + \frac{t}{t+m}.
\]
If $m< t/3$ then $t/(t+m) > t/(t+t/3)=3/4$, hence $\alpha>1+3/4=7/4$.
This contradicts $7/4^{+}$-power-freeness.  Therefore $m\ge t/3$, hence
$m\ge \lceil t/3\rceil$.
\end{proof}

In our setting, a witness template at position $j$ corresponds to a decomposition
of the edge-word as $w(e)=ABVB$ with $|B|=j$ and cost $r-j=|V|$.
Lemma~\ref{lem:74-cost} therefore implies that every admissible overlap at
position $j$ has cost at least $\lceil j/3\rceil$.
Combining this cost bound with Lemma~\ref{lem:rt-separation}
(available because $7/4^{+}$-free implies square-free) yields a uniform
bound on $\Omega_d(e)$.

\begin{lemma}\label{lem:Wd-74}
Let $d\ge 2$, and let $e$ be an edge whose edge-word $w(e)$ is unbordered and
$7/4^{+}$-power-free.  Then
\[
\Omega_d(e)\ \le\ \frac{4}{d-1}.
\]
\end{lemma}

\begin{proof}
Fix $j\le \lceil(D+1)/2\rceil$.  Every $r\in R_j(e)$ corresponds to a subword $BVB$
with $|B|=j$ and $|V|=r-j$, so Lemma~\ref{lem:74-cost} gives
$r-j\ge \lceil j/3\rceil$.  Moreover, since $W(e)$ is square-free,
Lemma~\ref{lem:rt-separation} implies that distinct elements of $R_j(e)$ are spaced by at least
$j+1\ge 2$.  Therefore
\[
\sum_{r\in R_j(e)} d^{-(r-j)}
\ \le\
\sum_{\ell\ge 0} d^{-(\lceil j/3\rceil + \ell (j+1)}
=
\frac{d^{-\lceil j/3\rceil}}{1-d^{-(j+1)}}
\ \le\
\frac{d^{-\lceil j/3\rceil}}{1-d^{-2}}
\ \le\
\frac{4}{3}\,d^{-\lceil j/3\rceil},
\]
using $d^{-2}\le 1/4$ for $d\ge 2$.
Summing over $j\ge 1$ gives
\[
\Omega_d(e)
\ \le\
\frac{4}{3}\sum_{j=1}^{\infty} d^{-\lceil j/3\rceil}.
\]
Finally, note that $\lceil j/3\rceil=m$ occurs for exactly $j=3m-2,3m-1,3m$, so
\[
\sum_{j=1}^{\infty} d^{-\lceil j/3\rceil}
=
\sum_{m=1}^{\infty} \bigl(d^{-m}+d^{-m}+d^{-m}\bigr)
=
3\sum_{m=1}^{\infty} d^{-m}
=
\frac{3}{d-1}.
\]
Hence $\Omega_d(e)\le (4/3)\cdot (3/(d-1))=4/(d-1)$.
\end{proof}

\subsection{Main theorem for all \texorpdfstring{$d\ge 2$}{d>=2} from $7/4^{+}$-free words}

\begin{theorem}[Main theorem for all $d$]
\label{thm:general-d-74}
Fix $d\ge 2$.  Let $e$ be any edge of $\K(d,D)$ whose edge-word $w(e)$ is an
unbordered $7/4^{+}$-power-free ternary word of length $D+1$, viewed inside $[d{+}1]$.
Then with the explicit constant
\[
C_d:=\frac{8}{d-1},
\]
we have
\[
U_D(e)\ \ge\ (D-C_d)\,d^{D-1}
\qquad\text{and}\qquad
\mathrm{cong}(e)\ \ge\ \frac{d}{d-1}\Bigl(1-\frac{1}{D(d-1)}\Bigr)(D-C_d)\,d^{D-1}.
\]
Consequently, for all $D\ \ge\ \frac{8d^2+2d-1}{d-1}$ we have $\mathrm{cong}(e) > \tau(d,D)$.
\end{theorem}

\begin{proof}
By Lemma~\ref{lem:Wd-74} we have the explicit bound
\[
\Omega_d(e)\ \le\ \Omega_d^\ast:=\frac{4}{d-1}.
\]
Lemma~\ref{lem:deficit-from-Wd} then gives
\[
U_D(e)\ \ge\ (D-2\Omega_d^\ast)\,d^{D-1}
\ =\ \Bigl(D-\frac{8}{d-1}\Bigr)d^{D-1},
\]
so we may take $C_d:=2\Omega_d^\ast=8/(d-1)$. Applying Corollary~\ref{cor:sumUk-from-UD} yields the stated lower bound on
$\mathrm{cong}(e)$.
We need a lower bound on $D$ so that

\begin{equation}\label{eq:D0-ineq}
\frac{d}{d-1}\Bigl(1-\frac{1}{D(d-1)}\Bigr)(D-C_d)\,d^{D-1}>\;\tau(d,D)=\left(D+\frac{D-1}{d}\right)d^{D-1}.
\end{equation}
Let $X:=(d-1)D$.  After some manipulation, we see that \eqref{eq:D0-ineq} is equivalent to the quadratic inequality
\begin{equation*}\label{eq:D0-quad}
X^2-\bigl(8d^2+2d-1\bigr)X+8d^2\ >\ 0.
\end{equation*}
In particular, it suffices to take
\begin{equation*}\label{eq:D0-suff}
X\ \ge\ 8d^2+2d-1,
\qquad\text{i.e.}\qquad
D\ \ge\ \frac{8d^2+2d-1}{d-1}.
\end{equation*}
\end{proof}

\begin{corollary}
\label{cor:D0-d2}
Let $d=2$, and let $e$ be an edge whose edge-word $w(e)$ is an unbordered
$7/4^{+}$-power-free ternary word of length $D+1$.
Then $\mathrm{cong}(e)>\tau(2,D)$ for all $D\ge 35$.
\end{corollary}

\begin{proof} This is Theorem~\ref{thm:general-d-74} with $d=2$.
\end{proof}

\begin{theorem}[Main theorem for $d=2$]
\label{thm:d-2}
For all $D\ge4$ there exists an edge $e \in E(K(d,D))$ such that $\mathrm{cong}(e)>\tau(d,D)$.
\end{theorem}
\begin{proof}
Corollary~\ref{cor:D0-d2} proves this statement for $D\ge35$.  For the remaining cases, 
we used a computer program to calculate the congestion of carefully selected edges.
Those results, which finish the proof, are displayed in Appendix~\ref{sec:smallD}.  
The maximum congestion for an edge in $K(2,3)$ is 15, which is less than $\tau(2,3)=16$; 
hence the stated bound is tight.
\end{proof}

\subsection{Corollary for \texorpdfstring{$d\ge 3$}{d>=3}: circular square-free words}

For $d\ge 3$ one can obtain a simpler (though less general in the hypothesis)
statement by assuming only circular square-freeness.

\begin{corollary}
\label{cor:general-d-csf}
Let $d\ge 3$, and let $e$ be an edge whose edge-word is a circular square-free
ternary word of length $D+1$ (viewed inside $[d{+}1]$).
Then there exists $D_0(d)$ such that for all $D\ge D_0(d)$ we have
\[
\mathrm{cong}(e)\ >\ \tau(d,D).
\]
\end{corollary}

\begin{proof}
This follows by repeating the uniform $\Omega_d$-estimate using only square-freeness
(and the corresponding separation lemma) together with the trimming inequality.
We omit the parallel bookkeeping since Theorem~\ref{thm:general-d-74} is the
stronger statement for all $d\ge 2$.
\end{proof}

\section{Computational evidence}
\label{sec:computations}

This section records numerical evidence supporting the two main empirical
observations that motivated our analytic approach:
\begin{enumerate}
\item circular square-free edge-words already force $\mathrm{cong}(e)>\tau(2,D)$
      for the small diameters we can exhaustively test when $D>3$ and
\item the weighted sparsity quantity $\Omega(e)$ appearing in
      Lemma~\ref{lem:weighted-sparsity} is very small for circular square-free
      (and related) edge-words, and appears to remain bounded as $D$ grows.
\end{enumerate}
All computations were carried out in $\K(2,D)$ 
to enumerate shortest paths and hence $\mathrm{cong}(e)$ exactly.

\subsection{Circular square-free edges: congestion statistics}

For each $D$ we enumerated all edge-words $w(e)$ of length $D+1$ that are
\emph{circular square-free} and computed $\mathrm{cong}(e)$ for each such edge $e$.

For $D=11$ we found:
\begin{itemize}
\item $72$ circular square-free edge-words in total;
\item $\mathrm{cong}(e)$ ranges from $18383$ to $19911$ over this class;
\item $\tau(2,11)=(3\cdot 11-1)2^{9}=16384$, so even the \emph{minimum} congestion
  in the circular square-free class exceeds $\tau(2,11)$ by about $12\%$.
\end{itemize}
The same phenomenon holds for $D=9,10$: every circular square-free edge tested
satisfies $\mathrm{cong}(e)>\tau(2,D)$.

\subsection{Weighted sparsity values $\Omega(e)$}

Lemma~\ref{lem:sufficient-W} shows that
\[
\Omega(e)\ <\ \frac{D-3}{8}
\qquad\Longrightarrow\qquad
\mathrm{cong}(e)>\tau(2,D).
\]

\medskip
\noindent
\textbf{Circular square-free class.}
For all circular square-free words for each $D$ from 20 to 40 and for a few randomly selected ones in each $D$ up to 70, we computed
$\Omega(e)$ and recorded
\[
\Omega_{\max}(D)\ :=\ \max\{\,\Omega(e):\ \Omega(e)\ \text{circular square-free of length }D+1\,\}.
\]
In all tested cases $\Omega(e)$ stays small (typically between $1$ and $1.5$), and
for every tested $D\ge 14$ we observed
\[
\Omega_{\max}(D)\ \le\ \frac{D-3}{8}.
\]
For smaller diameters ($D\le 13$) this inequality may fail, even though the
direct congestion computations above still show $\mathrm{cong}(e)>\tau(2,D)$ for
every circular square-free edge tested at $D=9,10,11$.  

\subsection{Square-free unbordered full-row edges}
\label{subsec:full-row-edges}

We also examined the ``full-row'' class at distance $D-2$, which provides a
different source of edges with high congestion.

\begin{definition}[Full-row edges at distance $k$]
Fix $D\ge 2$ and work in $\K(2,D)$.
For an edge $e$ and integers $1\le k\le D$, $1\le t\le k$, recall that
$N(e;k,t)$ denotes the number of ordered vertex pairs $(x,y)$ with
$\mathrm{dist}(x,y)=k$ whose unique geodesic from $x$ to $y$ uses $e$
in position $t$.
We say that $e$ is \emph{full-row at distance $k$} if
\[
N(e;k,t) = 2^{k-1}
\qquad\text{for every } t=1,\dots,k,
\]
that is,
\[
U_k(e) \;=\; \sum_{t=1}^k N(e;k,t) \;=\; k\,2^{k-1}.
\]
When $k = D-2$ we simply say that $e$ is \emph{full-row} (at distance $D-2$).
\end{definition}

For each $D$ we let $\mathcal F_D$ denote the edges that are full-row at distance
$D-2$, and
\[
\mathcal U_D := \{\, e \in \mathcal F_D : w(e)\text{ is unbordered} \,\}
\]
the unbordered full-row edges.  Motivated by the influence of internal
repetitions on geodesic traffic, we further define $\mathcal G_D$ to be the set
of those edges whose edge-words are simultaneously full-row at $D-2$, unbordered,
and square-free.

Table~\ref{tab:sqfree} summarizes these families for $D=9,10$.  In both cases the
set $\mathcal G_D$ is relatively small, but its members are exceptionally
congested: their average shortest-path congestion exceeds $\tau(2,D)$ by more
than $10\%$.

\begin{table}[h]
\centering
\begin{tabular}{@{}rrrrrr@{}}
\toprule
$D$ & $\tau(2,D)$ &
$|\mathcal F_D|$ &
$|\mathcal U_D|$ &
$|\mathcal G_D|$ &
$\dfrac{1}{|\mathcal G_D|}\displaystyle\sum_{e\in\mathcal G_D}\mathrm{cong}(e)/\tau(2,D)$ \\
\midrule
$9$  & $3328$ & $414$ & $222$ & $24$ & $1.1403$ \\
$10$ & $7424$ & $630$ & $240$ & $30$ & $1.1487$ \\
\bottomrule
\end{tabular}
\caption{Square-free unbordered full-row edges in $\K(2,D)$ for $D=9,10$.}
\label{tab:sqfree}
\end{table}

In both diameters tested, the maximum congestion over $\mathcal F_D$ is attained
by edges in $\mathcal G_D$, and the average congestion in $\mathcal G_D$ is
significantly larger than the average over $\mathcal U_D$ or $\mathcal F_D$.
This supports the general heuristic that forbidding borders and strong internal
repetitions concentrates geodesic traffic through a small set of edges.


\section{Open questions}
\label{sec:open-questions}

This paper shows that, for every outdegree $d\ge 2$ and diameter $D$ large enough,
some edge of $\K(d,D)$ must carry more shortest-path traffic than can be scheduled
within the regular routing makespan $\tau(d,D)$.  Our proofs use very special
families of edge-words (circular square-free, or circular $7/4^{+}$-free), and the
computations in Section~\ref{sec:computations} suggest that the set of edges with 
large congestion is much larger than the sets we employed.  Here we record a few 
natural directions for further study.

\subsection{Are circular square-free edges sufficient for $d=2$?}
\label{sec:open1}

The proof of the main theorem, Theorem~\ref{thm:general-d-74}, used the existence of
circular $\alpha$-power-free words for $\alpha<2$.  For $d=2$, we know of no other way to 
complete the proof, and yet our computations, as described in Section~\ref{sec:computations},
indicate that circular square-free words may be sufficient.  Can this be proven?
In particular: prove that for
$d=2$ and all sufficiently large $D$, there exists (or even that every) circular square-free
edge-word satisfies an inequality of the form $\Omega(e)\le c_0 + c_1 D$ with
$c_1<1/8$.  Any such bound, together with Lemma~\ref{lem:sufficient-W}, yields
$\mathrm{cong}(e)>\tau(2,D)$ for all large $D$.  That is, perhaps the main result
can be proven without the use of $\alpha$-power-free words, where $\alpha<2$.

\subsection{Full-row circular square-free edges}

In Section~\ref{subsec:full-row-edges} we introduced the notion of a \emph{full-row}
edge at distance $k$, and in particular at distance $D-2$ in $\K(2,D)$: an edge $e$ is
full-row at distance $D-2$ if $N(e;D-2,t)=2^{D-3}$ for every position $t$ and hence
$U_{D-2}(e)=(D-2)2^{D-3}$.  For the small diameters accessible to exhaustive search,
we observed a small family $\mathcal G_D$ of edges whose edge-words are simultaneously
unbordered, square-free, and full-row at distance $D-2$, and these edges had
noticeably higher congestion than typical edges.

\begin{quote}
\emph{Open question 1.}
For $d=2$, do there exist infinitely many diameters $D$ for which there is an edge
$e\in E(\K(2,D))$ whose edge-word is circular square-free and full-row at distance
$D-2$?  More broadly, for fixed outdegree $d\ge 2$, do there exist (for infinitely
many $D$) edges in $\K(d,D)$ whose edge-words are circular square-free and full-row at
some large distance $k$ (for example $k=D-2$)?
\end{quote}

\subsection{Density of high-congestion edges}

Let $\mathcal H_D\subseteq E(\K(d,D))$ denote the set of edges whose
shortest-path congestion exceeds the regular-routing makespan:
\[
\mathcal H_D
=
\bigl\{e\in E(\K(d,D)) : \mathrm{cong}(e)>\tau(d,D)\bigr\}.
\]
Our constructions produce explicit witnesses $e\in\mathcal H_D$ but do not address
how large $\mathcal H_D$ is.  For example, even when one restricts to circular
square-free ternary edge-words (embedded into $[d{+}1]$), it is known that the resulting set of edges
is exponentially sparse among all edges as $D$ grows (\cite{Currie,CurrieJohnson}).

\begin{quote}
\emph{Open question 2.}
For fixed $d\ge 2$, what can be said about the asymptotic density
\[
\frac{|\mathcal H_D|}{|E(\K(d,D))|}\,?
\]
Is this ratio bounded away from $0$ (or from $1$), does it tend to $0$ (or to $1$) as
$D\to\infty$, or does it fail to converge?
\end{quote}

\subsection{Capacity-boosted shortest paths}

Shortest-path routing has practical advantages (locality, simplicity), but as shown
here it cannot match the regular-routing makespan $\tau(d,D)$ under unit capacities
for all-to-all traffic.  A natural relaxation is to allow \emph{nonuniform}
capacities: assign each directed edge $e$ a capacity $c(e)\ge 1$ (packets per time
step), choose a shortest-path routing scheme, and measure the \emph{effective}
congestion
\[
\max_{e\in E(\K(d,D))}\frac{f_{\mathrm{SP}}(e)}{c(e)},
\]
where $f_{\mathrm{SP}}(e)$ is the number of chosen shortest paths using $e$.  One
would like to reduce the effective congestion to $\tau(d,D)$ while increasing total
capacity only on a negligible fraction of edges.

\begin{quote}
\emph{Open question 3.} (\cite{SkyFaberPC})
Fix $d\ge2$.  Does there exist, for each $D$, a subset $S\subseteq E(\K(d,D))$, edge
capacities $c(e)\ge 1$ with $c(e)>1$ only for $e\in S$, and a shortest-path routing
scheme such that
\[
\max_{e}\frac{f_{\mathrm{SP}}(e)}{c(e)}\ \le\ \tau(d,D)
\]
and
\[
\frac{1}{|E(\K(d,D))|}\sum_{e\in S}(c(e)-1)\ \longrightarrow\ 0
\qquad\text{as }D\to\infty?
\]
In other words, can one ``sprinkle bandwidth'' on a vanishing fraction of edges so
that some shortest-path scheme becomes schedulable within the regular-routing
makespan for all-to-all traffic?
\end{quote}

\section{Conclusion}
\label{sec:conclusion}

We expressed the shortest-path congestion $\mathrm{cong}(e)$ of an edge $e$ in
$\K(2,D)$ in terms of the counts $U_k(e)$ of geodesics of length $k$ using $e$,
and proved a trimming inequality showing that once $U_D(e)$ is
large, the total contribution from all shorter lengths is automatically large.

To bound $U_D(e)$ from below we introduced a template-based accounting of overlap
witnesses and bounded the resulting deficit $\Delta(e)$ in terms of the weighted overlap count $\Omega(e)$.
This reduces the $d=2$ comparison with the regular-routing makespan
$\tau(2,D)$ to obtaining a bound on $\Omega(e)$ for a suitable family of edge-words.
Our computations show that for circular square-free edge-words $\Omega(e)$ stays
small and, for all diameters tested, satisfies the sufficient inequality that
forces $\mathrm{cong}(e)>\tau(2,D)$ for $D>3$. However, we were not able to prove this in general and instead relied on the existence of $7/4^{+}$-power-free words. This
yields provably high-congestion edges for all but finitely many small diameters,
which can be handled by direct computation.

Finally, the same witness-template framework extends to $\K(d,D)$ for $d\ge 3$:
a fixed ternary edge-word can be embedded in the alphabet of size $d{+}1$, and
the overlap templates are unchanged while the number of completions scales with
$d$.
This gives an explicit family of edges whose shortest-path congestion exceeds
the regular-routing makespan $\tau(d,D)$ for all sufficiently large $D$ (in particular for all $D\ \ge\ \frac{8d^2+2d-1}{d-1}$), showing
that regular routing asymptotically beats every shortest-path routing scheme on
Kautz digraphs of any outdegree.


\section{Acknowledgments}
\label{sec:acknowledgments}

We would like to acknowledge the assistance provided by ChatGPT~\cite{openai2026chatgpt} 
in the generation of proof ideas, for making us aware of the literature 
concerning $7/4^{+}$-power-free words, and for producing an original draft 
of this paper.  We would also like to thank Amanda Streib for providing useful insights
related to the computations performed to collect the data shown in Appendix~\ref{sec:computations}.


\clearpage
\appendix

\section{Congestion for $d=2$ and small $D$}
\label{sec:smallD}

In Table~\ref{tab:d-2-smallD}, we show edges in $K(2,D)$ with larger congestion than $\tau(2,D)$ for
$4\le D\le 34$, completing the proof of Theorem~\ref{thm:d-2}.  When possible,
we have chosen a circular square-free edge as our representative.  For all other cases with $D\le15$, 
the edge shown has maximum congestion for $K(2,D)$.

\begin{table}[h]
\centering
\begin{tabular}{@{}rrrcc@{}}
\toprule
$D$ & $e$ \hspace{1.25in} & $\mathrm{cong}(e)$ & $\dfrac{\mathrm{cong}(e)}{\tau(2,D)}$ & circular square-free \\
\midrule
$4$  & 01210  & 45  & 1.023  \\
$5$  & 012102  & 113  & 1.009  & $\checkmark$  \\
$6$  & 0121020  & 299  & 1.099  &   \\
$7$  & 01210201  & 691  & 1.080  & $\checkmark$ \\
$8$  & 012102120  & 1753  & 1.191  &   \\
$9$  & 0120210201  & 3953  & 1.188  &   \\
$10$  & 01210212021  & 8559  & 1.153  & $\checkmark$  \\
$11$  & 012010212021  & 18383  & 1.122  & $\checkmark$ \\
$12$  & 0121021202102  & 42307  & 1.180  &  $\checkmark$ \\
$13$  & 01210201021012  & 96546  & 1.241  &   \\
$14$  & 010201210120212  & 197297  & 1.175  & $\checkmark$  \\
$15$  & 0121020102120102  & 431623  & 1.197  &  $\checkmark$ \\
$16$  & 01210201021010212  & 893023  & 1.160  &   \\
$17$  & 012021020102120102  & 1984207  & 1.211  &  $\checkmark$ \\
$18$  & 0102120121021202102  & 4218027  & 1.214  &  $\checkmark$ \\
$19$  & 01202120121021202102  & 9022627  & 1.229  &  $\checkmark$ \\
$20$  & 012102120102101202102  & 19337955  & 1.250  &  $\checkmark$ \\
$21$  & 0120210201210212012102  & 39896619  & 1.227  &  $\checkmark$ \\
$22$  & 01210212021020102120102  & 84835399  & 1.245  &  $\checkmark$ \\
$23$  & 010212010201210212012102  & 173173275  & 1.214  &  $\checkmark$ \\
$24$  & 0102120210201210212012102  & 367958155  & 1.236  &  $\checkmark$ \\
$25$  & 01021012010201210212012102  & 764885203  & 1.232  &  $\checkmark$ \\
$26$  & 010210120210201210212012102  & 1601680339  & 1.240  &  $\checkmark$ \\
$27$  & 0120212010212021020102120102  & 3374890767  & 1.257  &  $\checkmark$ \\
$28$  & 01021012102012101202120121012  & 6960049727  & 1.250  &  $\checkmark$ \\
$29$  & 010210120212010201210212012102  & 14422073299  & 1.249  &  $\checkmark$ \\
$30$  & 0121020102101201021202101201021  & 30100578799  & 1.260  &  $\checkmark$ \\
$31$  & 01201021012102120210201210212021  & 61391316303  & 1.243  &  $\checkmark$ \\
$32$  & 012021020121021201021012102120102  & 128665668047  & 1.261  &  $\checkmark$ \\
$33$  & 0102012021201020121012021020121012  & 262444521151  & 1.247  &  $\checkmark$ \\
$34$  & 01210201021012102120210201210212021  & 548535054079  & 1.265  &  $\checkmark$ \\
\bottomrule
\end{tabular}
\caption{Edges in $K(2,D)$ with $\mathrm{cong}(e) > \tau(2,D)$ for small $D$.}
\label{tab:d-2-smallD}
\end{table}

The graphs represented at the bottom of Table~\ref{tab:d-2-smallD} are quite large.  For example, 
the diameter $34$ graph has more than 25 billion vertices and more than 50 billion edges.  To make these calculations 
tractable, we never actually built these graphs.  We relied instead upon Lemma~\ref{lem:overlap-distance}, which 
can be used to compute shortest paths from the strings representing the vertices at the ends of the path.

\section{Appendix: Constructing circular $\alpha$-power-free edge-words}
\label{sec:construct}

In this section we sketch a practical procedure for generating the
edge-words used in the proof of Theorem~\ref{thm:general-d-74} (and its
$\alpha$-power variant).  The goal is: for a given diameter $D$ and a chosen
threshold $\alpha>1$, construct a word
\[
w = a_0a_1\cdots a_D \in \{0,1,2\}^{D+1}
\]
such that
\begin{enumerate}
\item[(i)] $a_i\ne a_{i+1}$ for all $i$ (Kautz constraint),
\item[(ii)] $w$ is $\alpha$-power-free as a linear word, and
\item[(iii)] $w$ is \emph{circular} $\alpha$-power-free, i.e.\ every cyclic rotation of $w$
  is $\alpha$-power-free.
\end{enumerate}
Any such word can be used as the edge-word of an edge $e$ in $\K(2,D)$
by taking
\[
u = a_0\cdots a_{D-1},
\qquad
v = a_1\cdots a_D.
\]

Two choices of $\alpha$ are of particular interest in this paper:
\[
\alpha=2 \quad\text{(square-free)}, 
\qquad\text{and}\qquad
\alpha=\tfrac{7}{4}^+ \quad\text{(i.e.\ $(7/4+\varepsilon)$-power-free for some fixed $\varepsilon>0$)}.
\]
In either case the same backtracking framework applies; only the local
repetition test changes.

We describe a backtracking procedure on a 6--state automaton closely
related to the $K(2,2)$ viewpoint of Shur~\cite{Shur}.

\subsection*{A 6--state automaton}

Because adjacent equal letters are forbidden, the local state of the
word is determined by the last two symbols:
\[
\Sigma = \{00,01,02,10,11,12,20,21,22\}
\]
but the Kautz constraint rules out $00,11,22$, leaving 6 legal states:
\[
S = \{01,02,10,12,20,21\}.
\]
We view $S$ as the vertex set of a directed graph $G$ in which there is
an edge $ab\to bc$ whenever $abc$ is a legal Kautz word of length~$3$ with 
$a,b,c\in\{0,1,2\}$.

A path
\[
s_0\to s_1\to\cdots\to s_{L-2}
\]
in $G$, with $s_i=a_i a_{i+1}$, encodes a word $a_0a_1\cdots a_{L-1}$
satisfying the Kautz constraint.  The condition that this word is
(circular) $\alpha$-power-free imposes additional forbidden patterns on
the path; the backtracking algorithm enforces these.

\subsection*{Online $\alpha$-power test}

Let $w=a_0a_1\cdots a_{k-1}$ be a partial word and let $\alpha>1$ be fixed.
When we append a letter $a_k$ we must ensure that no $\alpha$-power
appears as a \emph{suffix} of the new word.  Concretely, we check for subwords
of the form $x^m y$ occurring at the end of the word, where $x\neq\emptyset$,
$m\ge 1$, and $y$ is a (possibly empty) prefix of $x$, with
\[
\frac{|x^m y|}{|x|} \;\ge\; \alpha.
\]
Equivalently, writing $|x^m y| = m|x| + |y|$, this is
$|y|/|x| \ge \alpha - m$.

A simple (but effective) brute-force suffix test is:
\begin{itemize}
\item For each candidate period $p=1,2,\dots,k$ (the length $|x|$),
\item For each integer $m\ge 1$ with $mp \le k+1$,
\item Let $L:=k+1$ and let $q:=L-mp$ (so the last $mp$ letters are a candidate $x^m$);
  check whether the last $mp$ letters are $m$ copies of a length-$p$ block.
  If not, continue.
\item If yes, then let $0\le \ell < p$ be the largest prefix length such that the
  last $mp+\ell$ letters equal $x^m y$ with $y$ the prefix of $x$ of length $\ell$.
  (Equivalently, extend the match by comparing the next $\ell$ letters.)
\item Reject if $(mp+\ell)/p \ge \alpha$.
\end{itemize}

For $\alpha=2$ this reduces to the familiar ``square suffix'' test (with $m=2$ and
$\ell=0$).  For $\alpha=\tfrac{7}{4}^+$ one may fix a small $\varepsilon>0$ and
use $\alpha=7/4+\varepsilon$ in the inequality.

\medskip
\noindent
\textbf{Implementation note.}
For the diameters relevant in this paper, the naive test above is fast in practice.
If desired, one can speed up the checks using standard string tools (e.g.\ border tables,
prefix-function/KMP machinery) as in \cite{AlgorithmsOnStrings}.

\subsection*{Cyclic closure test}

To obtain a circular $\alpha$-power-free word of length $D+1$, we perform
backtracking up to length $D+1$ and then test that \emph{all} cyclic
rotations of the completed word are $\alpha$-power-free.

Given a candidate word $w=a_0\cdots a_D$ produced by the online test,
form the cyclic word
\[
\widetilde{w} = a_0a_1\cdots a_D a_0a_1\cdots a_D
\]
of length $2(D+1)$ and check that no $\alpha$-power subword $x^m y$ occurs in
$\widetilde{w}$ with its starting position in $\{0,1,\dots,D\}$.
This is equivalent to requiring that every cyclic rotation of $w$ is
$\alpha$-power-free.

\subsection*{Backtracking algorithm}

We now summarize the full construction.

\begin{enumerate}
\item Fix $D\ge 1$ and set $L=D+1$ and a threshold $\alpha>1$.
\item Choose an initial pair $s_0=a_0a_1\in S$.
\item Recursively build a path
\[
s_0\to s_1\to\cdots\to s_{L-2}
\]
in the automaton $G$, where $s_i=a_i a_{i+1}$, by trying all legal
successors $s_{i+1}$ of $s_i$ (equivalently, choose $a_{i+1}\ne a_i$ and
$a_{i+2}$ is then determined by the chosen state).  At each step:
  \begin{itemize}
  \item append the new letter $a_{i+2}$;
  \item apply the online $\alpha$-power test to the current prefix;
  \item if the test fails, backtrack to try another successor.
  \end{itemize}
\item When the path reaches length $L-1$ (so $w=a_0\cdots a_D$ has length $L$),
  apply the cyclic closure test.  If it passes, we have constructed a circular
  $\alpha$-power-free word of length~$L$.  If it fails, backtrack and continue.
\end{enumerate}

\subsection*{Existence and the two choices $\alpha=2$ and $\alpha=\tfrac{7}{4}^+$}

For $\alpha=2$, the existence of ternary circular square-free words of length $n$
is treated in \cite{Currie,Shur,CurrieJohnson}.
For $\alpha=\tfrac{7}{4}^+$, there are also existence results for circular ternary
$7/4^{+}$-free words (equivalently, circular $(7/4+\varepsilon)$-power-free words for
some fixed $\varepsilon>0$); in particular, one has existence for all sufficiently
large lengths (see \cite{CMR_CRT_2019}).

\subsection*{From words to edges}

Once a circular $\alpha$-power-free word $w=a_0a_1\cdots a_D$ has been found, we
may rotate it so that the resulting edge-word is convenient for computation (for example,
fixing $a_0=0$ and $a_1=1$).  We then define $u = a_0a_1\cdots a_{D-1}$ and 
$v = a_1a_2\cdots a_D$ and take $e:u\to v$ as the edge in $\K(2,D)$ whose congestion is analyzed in the
main text.

\subsection*{Remarks}

\begin{itemize}
\item The backtracking algorithm described above is sufficient to generate examples
  for the diameters we need in practice.  It also provides a convenient way to explore
  the witness templates experimentally for larger $D$.

\item For a more structural construction, one can instead use the uniform morphisms and
  ``level'' constructions described in \cite{Currie,Shur,CurrieJohnson} (for $\alpha=2$),
  and analogous morphic constructions for repetition-avoidance thresholds (see
  \cite{CMR_CRT_2019} and references therein) for $\alpha=\tfrac{7}{4}^+$.
\end{itemize}

\end{document}